\documentclass{amsart}
\usepackage{mathrsfs}
\usepackage{epsfig}
\usepackage{graphicx}
\usepackage{amsthm,amsmath,amsfonts,amssymb,amscd}
\usepackage{cancel}
\usepackage{hyperref}

\usepackage{enumerate}
\def\AA{{\sf A}}

\def\JJ{{\sf {J}}}

\def\1{{\sf 1}}\def\GG{{\sf {G}}}
\def\0{{\sf 0}}
\def\e{{\sf e}}
\def\a{{\sf a}}
\def\b{{\sf {b}}}

\def\c{{\sf {c}}}\def\d{{\sf {d}}}
\def\w{{\sf {w}}}
\def\v{{\sf v}}

\def\z{{\sf z}}
\def\g{{\sf g}}
\def\h{{\sf h}}
\def\II{{\sf I}}
\def\so{{\sf supp}}

\def\qed{\hfill {\hbox{\footnotesize{$\Box$}}}}
\def\L{{\mathscr  L}}
\def\P{{\mathscr  P}}
\def\B{{\mathcal B}}
\def\C{{\mathscr C}}
\def\M{{\mathcal M}}

\def\RR{\mathbb{R}}

\def\qed{\hfill {\hbox{\footnotesize{$\Box$}}}}
\def\0{{\sf {\sf 0}}}

\def\a{{\sf a}}

\def\w{{\sf w}}
\def\y{{\sf y}}

\def\v{{\sf v}}
\def\0{{\sf 0}}

\def\so{{\sf supp}}

\def\proof{{\noindent\bf Proof.}\hskip 0.3truecm}
\def\BBox{\kern  -0.2cm\hbox{\vrule width 0.15cm height 0.3cm}}

\def\AA{{\sf A}}

\def\B{\mathcal{B}}
\def\E{\mathcal{E}}

\def\DD{{\sf D}}
\def\F{\mathcal{F}}
\def\G{\mathcal{G}}

\def\I{\mathcal{I}}

\def\K{\mathcal{K}}
\def\L{\mathcal{L}}
\def\M{\mathcal{M}}
\def\P{\mathcal{P}}

\def\GG{{\sf G}}
\def\II{{\sf I}}
\def\JJ{{\sf J}}

\def\LL{{\sf L}}
\def\MM{{\sf M}}

\def\RR{\mathbb{R}}

\def\BBox{\kern  -0.2cm\hbox{\vrule width 0.15cm height 0.3cm}}

\oddsidemargin .cm \evensidemargin 1.cm \textwidth=16.5cm
\graphicspath{{../Portada/}}

\include {mak}
\begin{document}

\newtheorem{propo}{Proposition}[section]
\newtheorem{lemma}[propo]{Lemma}
\newtheorem{theorem}[propo]{Theorem}
\newtheorem{corollary}[propo]{Corollary}
\newtheorem{definition}[propo]{Definition}
\newtheorem{ex}[propo]{Example}
\newtheorem{rem}[propo]{Remark}

\def\1{{\sf 1}}
\def\0{{\sf 0}}
\def\e{{\sf e}}
\def\a{{\sf a}}
\def\v{{\sf v}}
\def\y{{\sf y}}
\def\z{{\sf z}}
\def\g{{\sf g}}
\def\h{{\sf h}}
\def\II{{\sf I}}
\def\JJ{{\sf J}}
\def\so{{\sf supp}}
\def\proof{{\noindent  Proof.}\hskip 0.3truecm}
\def\BBox{\kern  -0.2cm\hbox{\vrule width 0.15cm height 0.3cm}}
\def\L{{\mathcal  L}}\def\P{{\mathcal  P}}
\def\B{{\mathcal B}}
\def\K{{\mathcal K}}\def\M{{\mathcal M}}
\def\E{{\mathcal E}}
\def\F{{\mathcal F}}
\def\G{{\mathcal G}}
\oddsidemargin 16.5mm
\evensidemargin 16.5mm

\thispagestyle{plain}

\vspace{5cc}
\begin{center}

{\Large\bfseries
The $M$-matrix group inverse problem for recoverable complete networks
}

\vspace{1em}
\renewcommand{\thefootnote}{\fnsymbol{footnote}}
{\large\itshape
\'Angeles Carmona\footnotemark[1],
Andr\'es M. Encinas\footnotemark[1],
Sweta Patra\footnotemark[2],
K.C. Sivakumar\footnotemark[2]
}
\footnotetext[1]{Departament de Matem\`atiques, Universitat Polit\`ecnica de Catalunya, 08034 Barcelona, Spain 
(angeles.carmona@upc.edu, andres.marcos.encinas@upc.edu)}

\footnotetext[2]{Department of Mathematics, Indian Institute of Technology Madras, Chennai 600036, India (swetapatra.math@gmail.com, kcskumar@iitm.ac.in)}

\vspace{1em}

\parbox{0.7\textwidth}{
\small
\noindent\textbf{Abstract.}
This study investigates the conditions under which the group inverse of a singular, irreducible, symmetric $M$-matrix retains the $M$-matrix property. By concentrating on a structured subclass derived from rank-one perturbations of a diagonal matrix, and inspired by recoverable complete networks, we obtain explicit analytical results. Utilizing both matrix-theoretic methodologies and potential theory on networks, we establish necessary and sufficient conditions for the $M$-property of the specified network in terms of conductances and associated Doob potentials. This framework facilitates the construction of families of singular, irreducible $M$-matrices whose group inverses maintain the $M$-matrix structure. Our findings offer novel insights into this research domain and enhance the relationship between $M$-matrix theory and network analysis.

\vspace{0.5em}


}

\end{center}

\vspace{1cc}

\section{Introduction}

In the framework of $M$-matrix theory, one of the central and long-standing problems is to determine the conditions under which the group inverse of a singular, irreducible, symmetric $M$-matrix remains an $M$-matrix. 
Let $\M_n(\RR)$ denote the set of all real $n\times n$ matrices.
A matrix ${\sf A}\in \M_n(\RR)$ is called a $Z$-matrix if its off-diagonal entries are nonpositive. Any such matrix can be expressed in the form ${\sf A}=s{\sf I-B}$, $s\in \RR$ and $\sf B\ge 0$ (i.e. all the entries of $\sf B$ are nonnegative). When $s\ge \rho(\sf B)$ ({\it spectral radius} of $\sf B$), $\sf A$ is an $M$-matrix. There are many equivalent characterizations of $M$-matrix listed in the book \cite{BP1994}. Among these, a fundamental result for the symmetric case states that a symmetric $Z$-matrix is an $M$-matrix iff it is positive semidefinite.

Throughout this paper, a network (or graph) is understood to be finite, connected, and simple; that is, it contains neither loops nor multiple edges.
Since the off diagonal entries of a $Z$-matrix are non positive, each irreducible and symmetric $Z$-matrix can be identified as a Schr\"odinger operator defined on a finite, connected network. 
A Schr\"odinger operator on a network is obtained by adding some values, called potentials,  to the diagonal of its combinatorial Laplacian. 
In the matrix representation of a Schr\"odinger operator on a finite network, each off-diagonal entry is nonzero (and therefore negative) iff the corresponding vertices are adjacent.
 In this case, the absolute value of the off-diagonal entry of the matrix is the conductance of the edge joining those two adjacent vertices. Moreover, the irreducibility of the matrix is equivalent to the connectedness of the underlying network.
If the matrix associated with a Schr\"odinger operator is positive semidefinite, then it is an $M$-matrix. Within this framework, the group inverse of the matrix associated with a Schr\"odinger operator corresponds to the matrix associated with the Green function. This naturally leads to the problem of determining the conditions under which the matrix associated with the Green function is an $M$-matrix.
Equivalently, we characterize when the group inverse of a singular, positive semidefinite Schr\"odinger operator on a finite, connected network is also a singular, positive semidefinite Schr\"odinger operator on another suitable network. 

In this work, we are interested in a structured class of singular, irreducible $M$-matrices arising from a distinguished family of complete networks that are recoverable. We characterize the class of singular, irreducible, symmetric $M$-matrices arising from conductances and positive potentials associated with a singular, positive semidefinite Schr\"odinger operator on a recoverable complete network. We prove that every matrix in this class admits a remarkably simple structure: it can be written as a rank-one perturbation of a positive diagonal matrix.
This structural characterization yields an explicit formula for the group inverse, which is equivalent to computing the Green function of the underlying network.

The main contribution of this study is to establish the necessary and sufficient conditions under which the group inverse of such matrices preserves the $M$-matrix structure. From another perspective, we characterize when a recoverable complete network satisfies the $M$-property with respect to a positive function.
Several illustrative examples and counterexamples are presented to demonstrate the sharpness of our results and to clarify the limitations of $M$-structure preservation in this setting.

The idea of recoverable networks is adopted from electrical network theory. The motivation is that, for such networks, their edge conductances admit a multiplicative representation, resulting in a reduction in the number of free parameters, from $\frac{n(n-1)}{2}$ to $n$ ($n\ge 4)$. Despite this reduction, the graph remains complete and thus preserves the adjacency among all vertices. Furthermore, the matrix corresponding to the singular, positive semidefinite Schr\"odinger operator on an $n$-vertex recoverable complete network, can be expressed as a positive scalar multiple of the Schur complement of the first diagonal entry (nonzero) of the matrix corresponding to the singular, positive semidefinite Schr\"odinger operator on the star network on $n+1$ vertices.

We conclude this introduction with a brief survey of related work. Suppose $\sf B\in \M_n(\mathbb{R})$ is an (entrywise) nonnegative and irreducible matrix. Then, the matrix $\sf A=\rho(B)I-B$ is a singular, irreducible $M$-matrix.
In \cite{NPW1982}, Neumann et al. first raised the question of when the group inverse of $\sf A$, that is, $\sf A^\#$ is an $M$-matrix.
Then, it has been shown that if $\sf B$ is a rank-one matrix, then $\sf A^{\#}$ is an $M$-matrix \cite{DN1984}. This was further investigated in \cite{KN1995*}, where the case of the matrix $\sf B$ having only a few distinct eigenvalues was considered. Among other things, they presented a result in which the $M$-matrix $\sf A$ satisfies $\sf A^\#=\beta A^\top$ for some $\beta>0$. In \cite{CN1997}, it was shown that if $\sf B$ is a stochastic matrix, then $\sf (I-B)^\#$ is an $M$-matrix only when $\sf B$ lies in a small wedge of a rank-one nonnegative matrix. We refer the reader to \cite{KLS2021} for results that generalize some of these earlier findings.
Recently, Carmona et al. \cite{CPS2026} investigated real matrices of the form $\sf A= (1-\rho)I+\rho uv^\top$, where $\sf u,v\in \RR^n$ satisfy $\sf u,v\ge e$ (entrywise) and $\sf \rho<0$, where  $\sf e$ denotes the all-ones vector. It was shown that both $\sf A$ and its group inverse $\sf A^\#$ are singular, irreducible $M$-matrices. These results further motivate the questions considered in this work.

\section{Preliminaries}
Let $(V,E)$ be a finite, connected graph, without loops or multiple edges, where $V=\{x_1,\ldots,x_n\}$ and $E$ are the sets of vertices and edges, respectively. The set of real valued functions on $V$ is denoted by $\C(V)$ and it is endowed with the standard inner product $\langle\cdot,\cdot\rangle$, defined by $\langle u,v\rangle=\sum\limits_{x\in V}u(x)\,v(x)$; $u,v\in \C(V)$. In addition, the norm of $u\in \C(V)$ is defined as $||u||=\big(\sum\limits_{x\in V}u(x)^2\big)^{\frac{1}{2}}.$
Given $u\in \C(V)$, the subspace of $\C(V)$ that is orthogonal to $u$ is denoted as  $u^\bot$. Moreover, for any $x\in V$, $\varepsilon_x$ stands for the {\it Dirac function}, which is defined as follows: $\varepsilon_x(y)=1$ when $y=x$ and $\varepsilon_x(y)=0$ when $y\neq x;~ y\in V.$ Note that, the set $\{\varepsilon_x:x\in V\}$ is a basis for the vector space $\C(V)$. A function $u\in \C(V)$ is called a positive function if $u(x)>0$ for every $x\in V.$ Such a function is called a {\it weight} whenever $||u||=1$ and
the set of weights is denoted by $\Omega(V)$.

A {\it network} $\Gamma$ with underlying graph $(V,E)$ and {\it conductance} $c$ is the tuple $(V,E,c)$ where 
$c\colon V\times V\longrightarrow [0,\infty)$  is a symmetric function, i.e. $c(x,y)=c(y,x)$ for any  $x,y\in V$, satisfying that $c(x,y)>0$ iff $x$ is adjacent to $y$, i.e. $x\sim y$. In particular, $c(x,x)=0$ for any $x\in V$.
The {\it degree function of the network}
$\Gamma$ is the function $\kappa\in \C(V)$ defined by $\kappa(x)=\sum\limits_{y\in V}c(x,y),\; x\in V$. If $c(x,y)=1$ for $x\sim y$, then $\kappa(x)$ is  the number of vertices adjacent to $x$. 
 
The {\it combinatorial Laplacian}, or simply the {\it Laplacian}, of the network $\Gamma$ is the endomorphism  on $\C(V)$ that assigns to any $u\in\C(V)$ the function  \begin{equation}
\label{laplaciano}\L(u)(x)\displaystyle=\sum\limits_{y\in V}
c(x,y)\big(u(x)-u(y)\big)=\kappa(x)u(x)-\sum\limits_{y\in V}
c(x,y)\,u(y),\hspace{.25cm}x\in V.
\end{equation}
It is well known, and easy to prove, that the  Laplacian of a network is singular, self-adjoint (with respect to the inner product $\langle \cdot,\cdot\rangle$) and positive semidefinite. Moreover, $\L(u)=0$ iff $u$ is a constant function, because the network is connected.

For any $q\in\C(V)$, the {\it Schr\"odinger operator} on $\Gamma$ with potential $q$ is the endomorphism on $\C(V)$ that assigns to any $u\in\C(V)$ the function $\L_q(u)=\L (u) +qu$. 

For $\omega\in \Omega(V)$, the function $q_{\omega}=-\omega^{-1}\L(\omega)$ is called the {\it (Doob)-potential determined by $\omega$}. Clearly, $q_\omega=0$ iff $\omega$ is constant and the  Schr\"odinger operator with potential $q_\omega=0$ is the Laplacian. More generally, since $\langle \omega,q_\omega\rangle=0$, we infer that $q_\omega$ takes positive and negative values, except when $\omega$ is a scalar multiple of the constant function. On the other hand, from the identity \eqref{laplaciano}, we deduce that 
\begin{equation}
\label{potencial}
q_\omega(x)=-\kappa(x)+\omega(x)^{-1}\sum\limits_{y\in V}c(x,y)\omega(y), \hspace{.25cm}\hbox{for any $x\in V$.}
\end{equation}

The following result characterizes all the singular, positive semidefinite Schr\"odinger operators on $\Gamma$, see  \cite[Proposition 3.3]{BCE05}. 

 \begin{propo}
\label{positividad} Given  $q\in \C(V)$, the Schr\"odinger operator  $\L_q$ on $\Gamma$ is singular, positive semidefinite iff there exists $\omega\in \Omega(V)$  such that   $q=q_\omega$. Moreover, $ \omega$ is uniquely determined and $\L_{q_\omega}(u) =0$ iff  $u=a\omega$, $a\in \RR$. 
\end{propo}
In light of the above proposition, we remark that for a given $\omega\in \Omega(V)$, the associated Schr\"odinger operator $\L_{q_\omega}$ is a singular, positive semidefinite operator.

\section{Networks having the $M$-property, with respect to a weight}

For the network $\Gamma$ and the weight $\omega \in \Omega(V)$,  the properties of the Schr\"odinger operator $\L_{q_\omega}$ determine that given  $f\in \C(V)$, the  {\it Poisson equation}  $\L_{q_\omega}(u)=f$  has solution iff $f\in\omega^\bot$. In this case, there exists a unique solution such that $u\in\omega^\bot$, see \cite[Proposition 2.1]{CEM2014}. Therefore, $\L_{q_\omega}$ is an isomorphism on $\omega^\bot$ whose inverse is called the {\it Green operator of $\Gamma$ with respect to $ \omega$}. 

Let $\P_\omega:\C(V)\longrightarrow \C(V)$ denote the projection on ${\rm span}\{\omega\}$ and is defined as  $\P_\omega(f)=\langle f,\omega\rangle\omega$. Then, the Green operator with respect to $\omega$ can be identified with the endomorphism $\G_\omega:\C(V)\longrightarrow \C(V)$ 
assigning to any $f\in \C(V)$ the unique solution of the Poisson equation with data $f-\P_\omega(f)$ belonging to $\omega^\bot$. 

It is easy to prove that $\G_\omega$ is a singular, self-adjoint, positive semidefinite  operator. Moreover, $\G_\omega(u)=0$ iff $u=a\omega$, $a\in \RR$ and 
\begin{equation}
\label{grupo}
\L_{q_\omega}\circ \G_\omega=\G_\omega\circ\L_{q_\omega}=\I-\P_\omega,
\end{equation}
where  $\mathcal{I}$ is the identity operator. 

The {\it Green function with respect to $\omega$} is $G_\omega:V\times
V\longrightarrow\mathbb{R}$ defined as $G_\omega(x,y)=\langle \G_\omega(\varepsilon_y),\varepsilon_x\rangle $, \;$x,y\in V$. Since for $x\in V$, $\varepsilon_x$ cannot be a multiple of the weight $\omega$, we get that $G_\omega(x,x)=\langle \G_\omega(\varepsilon_x),\varepsilon_x\rangle >0$ for any $x\in V$. In addition, given $y\in V$, the function $u=G_\omega(\cdot,y)$ is characterized as the unique solution of the Poisson equation 
$$\L_{q_\omega}(u)=\varepsilon_y-\omega(y)\omega,$$ 
belonging to $\omega^\bot$.

\begin{rem}
We can extend the definition of the Doob potential, determined by the positive function $\sigma\in \C(V)$ as $q_\sigma=-\sigma^{-1}\L(\sigma)$. Clearly, $\omega =||\sigma||^{-1}\sigma\in \Omega(V)$ and  $q_\sigma=q_\omega$. Thus,  $\L_{q_\sigma}=\L_{q_\omega}$.
In addition, ${\rm span}\{\sigma\}={\rm span}\{\omega\}$ and $\sigma^\bot=\omega^\bot$. For any $u\in \C(V)$, the projection of $u$ onto $\omega^\bot$ is given by $$\operatorname{proj}_{\omega^\bot}(u)=u-\langle u,\omega\rangle \omega=u-||\sigma||^{-2}\langle u,\sigma\rangle \sigma.$$ Thus, in order to express $G_\omega$ in terms of $\sigma$, we need to solve the Poisson equation
$\L_{q_\sigma}(u)=\varepsilon_y-||\sigma||^{-2}\sigma(y)\sigma$.
\end{rem}

When $\omega\in \Omega(V)$ is constant, we omit the dependence of $\omega$ and remove the corresponding subindex. Therefore, the Green operator and Green function of a given network are considered to be associated with the combinatorial Laplacian.

Next, we recall the  $M$-property definition, introduced in \cite{CEPS2026}.

\begin{definition}
Given $\omega\in \Omega(V)$, the network $\Gamma$ has the $M$-property with respect to $\omega$ if  $G_{\omega}(x,y)\le 0$ for any $x,y\in V$ with $x\not=y$. 
\end{definition}
If $\omega$ is constant and the network $\Gamma$ has the $M$-property with respect to $\omega$, then we say that the network $\Gamma$ has the $M$-property.

The {\it minimum principle}, which is satisfied by any positive semidefinite Schr\"odinger operator (see \cite[Proposition 4.4]{BCE05}), implies that it suffices to know the values of the Green function at adjacent vertices to conclude that the $M$-property holds, see \cite{KN98,KN2011} for the case of trees. Thus, we recall the following theorem, which was mentioned in \cite{CEPS2026}.
\begin{theorem}
The network $\Gamma$ has the $M$-property with respect to   $\omega \in \Omega(V)$ iff  $G_\omega(x,y)\le 0$ for any  $x\sim y$. 
\end{theorem}

Let $\LL_q$ be the matrix associated with the Schr\"odinger operator $\L_q$ with potential $q\in \C(V)$ on $\Gamma$. Then, $\LL_q$ is an irreducible, symmetric $Z$-matrix of order $n$ whose off diagonal entries are $-c(x_i,x_j)$, $i,j=1,\ldots,n,$ $i\not=j$, and whose diagonal entries are  given by $\kappa(x_i)+q(x_i)$, $i=1,\ldots,n$. 
With this identification, Proposition \ref{positividad} establishes that  $\LL_q$ is a singular, irreducible $M$-matrix iff $q=q_\omega$ for some weight $\omega\in \Omega(V)$. 

Consider the matrix $\GG_\omega$ whose $ij$  entry is given by $G_\omega(x_i,x_j)$. The identification between the Green function $G_\omega$ and the matrix $\GG_\omega$ shows that $\GG_\omega$ is an irreducible, singular, symmetric, and positive semidefinite matrix. Moreover, the identity \eqref{grupo} implies that  $\GG_\omega=\LL_{q_\omega}^\#$, the group inverse of  $\LL_{q_\omega}$.

Therefore, the network $\Gamma$ satisfies the $M$-property with respect to $\omega\in \Omega(V)$ iff  $\LL_{q_\omega}^\#$ is an $M$-matrix. In particular, $\Gamma$ satisfies the $M$-property iff the group inverse of the matrix associated with the combinatorial Laplacian is an $M$-matrix.

The problem of determining irreducible singular $M$-matrices, whose group inverses are also $M$-matrices, has been investigated only in specific cases, despite its relevance in many applications. Existing results are largely confined to the combinatorial Laplacian, which corresponds to a singular, symmetric, diagonally dominant $M$-matrix. Consequently, the non-diagonally dominant case remains largely unexplored. 

For instance, Chen {\it et al.} obtained in \cite{CKN95} that a {\it weighted path}, viz., a network whose underlying graph is a path, satisfies the $M$-property only when the number of vertices is less than $4$. Furthermore, for paths with length $3$ or $4$, the $M$-property holds only under severe restrictions on the conductances. More generally, Kirkland {\it et al.} proved in \cite{KN98} that the only weighted trees, viz., networks whose underlying graph is a tree, satisfying the $M$-property are  paths of length less than $4$ and stars whose conductances satisfy appropriate constraints. In this case, they concluded that for any size, there exist infinitely many star networks satisfying the $M$-property.

The introduction of {\it Potential Theory} into this problem, and more specifically the study of singular, positive semidefinite Schr\"odinger operators on networks, allows us to treat the problem in its full generality, albeit restricted to the symmetric case. Recently, in \cite{CEPS2026}, the authors studied the class of conductances for which the network associated with a star graph has the $M$-property with respect to a weight. In addition, in \cite{BCEM2012}, it was proven that for any  nonnegative integer $n$, there exist infinitely many weights and infinitely many conductances such that the corresponding path with these conductances has the $M$-property with respect to these weights. From the matrix point of view, this means that unlike the non-existence of singular, irreducible, symmetric, tridiagonal, and diagonally dominant $M$-matrices of order greater than $4$, if we eliminate the diagonal dominance hypothesis, then there exist infinitely many families of such matrices for any fixed order. However, a complete description of all matrices having this property remains an open problem. As an indication of its difficulty, it suffices to observe that for $n=4$, $10$ infinite families of such matrices were explicitly determined in \cite{CEM2013}; however, they do not cover all the possibilities. Even in the case of the combinatorial Laplacian of graphs, the problem remains unsolved. For example, for symmetric structures such as distance-regular  or distance-biregular graphs, the problem has been solved,  see \cite{ACEJ23,BCEM2012*,KP2013}, but the general case remains open.

\section{Recoverable complete networks and the $M$-property}\label{recoverable complete}

The so-called neighborhood transformation, applied at the center of a star network with 
$n+1$ vertices, yields an electrically equivalent complete network on 
$n$ vertices. To proceed, we recall the definition of neighborhood transformation  introduced in \cite{CEM2019}.
Since a finite network is completely determined by its vertex set and the associated conductance function, we may consider a network as a pair consisting of these two objects. 

Let $\Tilde{\Gamma}=(\Tilde{V},\Tilde{c})$ be a given network. We fix a vertex $x_0\in \Tilde{V}$ and set $F=\Tilde{V}\setminus \{x_0\}$.
Define a new conductance $c^{x_0}:F\times F\to [0,\infty)$ by $c^{x_0}(x,x)=0,$ $x\in F$ and
$$c^{x_0}(x,y)=\Tilde{c}(x,y)+\frac{\Tilde{c}(x,x_0)\Tilde{c}(x_0,y)}{\kappa(x_0)},\hspace{0.15cm} x,y\in F, \; x\ne y,$$
where $\kappa(x_0)=\sum\limits_{z\in \Tilde{V}} \Tilde{c}(x_0,z)$ is the degree at $x_0$.

Thus, we obtain a new network $\Gamma^{x_0}=(F,c^{x_0})$ from the given network $\Tilde{\Gamma}$. The adjacency relation in $\Gamma^{x_0}$ is defined by 
$$x,y\in F,\;  x\sim y \;\text{iff}\;  c^{x_0}(x,y)>0.$$ This transformation is called {\it neighborhood transformation} at the vertex $x_0.$  Note that, the connectedness of the new network is still maintained after the neighborhood transformation at $x_0.$ 

Applying the neighborhood transformation at the center $x_0$ of a star subnetwork of $\tilde{\Gamma}$ yields a network $\Gamma^{x_0}$ in which the subnetwork induced on $N(x_0)$ (the set of vertices adjacent to $x_0$) is complete. We refer to this procedure as the {\it star-complete transformation}. In the case $|N(x_0)| = 3$, it reduces to the classical $Y\!-\!\Delta$ transformation.

We next recall \cite[Proposition 8.1]{CEM2019}, which motivates and formalizes the notion of a recoverable complete network. This result shows that a complete network can be obtained from a star network via the neighborhood transformation. We state the proposition below and then introduce the corresponding definition.

\begin{propo}\label{recov}
Let $\Tilde{ \Gamma}=(\Tilde{V},\Tilde{ c})$ be a network. Let $K$ be the vertex set of a complete graph with $K\subseteq\Tilde{V}$.
     Then,  the complete network $\Gamma=(K,\Tilde{c})$ is obtained from a star network by a star-complete transformation at the center, iff there exists a nonnegative  function $\alpha\in \C(\Tilde{V})$ with $\text{supp}(\alpha)=K$ and $\Tilde{ c}(x,y)=\alpha(x)\alpha(y)$, $x,y\in K, \; x\ne y.$
\end{propo}

For the remainder of this section, let $\Gamma=(V,\widehat c)$ be a complete network where $V=\{x_1,\ldots,x_n\}$ is the vertex set and $\widehat c:V\times V \to [0,\infty)$ is the conductance function, respectively.

\begin{definition}
The complete network $\Gamma=(V,\widehat c)$ is recoverable iff there exists a positive function $c\in \C(V)$ such that $\widehat c(x,y)=c(x)c(y)$ for any $x,y\in V$, $x\not=y$.   
\end{definition}

If $|V|=3$, then the complete network $\Gamma$ is recoverable for any conductance $\widehat c$.
However, when $|V|\ge 4$, there exist conductances  for which $\Gamma$ is not recoverable. In fact, it was shown in \cite[Proposition 8.5]{CEM2019} that for $|V|\ge 4$, the complete network $\Gamma=(V,\widehat c)$ is recoverable if and only if, for any four distinct vertices $x,y,u,v\in V$, it is satisfied that $\widehat c(x,y)\widehat c(u,v)=\widehat c(x,u)\widehat c(v,y)$.
Moreover, any recoverable complete network with more than two vertices is uniquely recoverable. In contrast, a digon (the complete network on two vertices) is infinitely recoverable; see \cite[Corollary 8.2]{CEM2019}.

Throughout this section, we consider a positive function $\omega \in \mathcal{C}(V)$, not necessarily normalized, and the associated singular positive semidefinite Schr\"odinger operator $\mathcal{L}_{q_\omega}$ on $\Gamma$.

The following notations are used in this section.
We denote $c_j=c(x_j), w_j=\omega(x_j)$; $j=1,\ldots,n$. Let $\c=(c_1,\ldots,c_n)^\top\in \RR^n$ and $ \w=(w_1,\ldots,w_n)^\top \in \RR^n$. 
The notation $u>0$ indicates $u\in \RR^n$ as an entrywise positive vector. 

Thus, the conductance on the edge  $\{x_i,x_j\}$ is given by $c_ic_j$ for $i,j=1,\ldots,n,\; i\ne j$. Note that the positive vector $\w$ is the eigenvector corresponding to the eigenvalue $0$ of the matrix associated with the Schr\"odinger operator $\L_{q_\omega}$ on $\Gamma$. 

Then, all singular, symmetric and irreducible $M$-matrices supported by a recoverable complete network $\Gamma$ on $n$ vertices are given by the matrix 
\begin{equation}
\label{Schr2}
\widehat \LL(\c,\w)=\begin{bmatrix}  \frac{\alpha c_1}{w_1}-c_1^2 & -c_1c_2 & \cdots & -c_1c_n\\
-c_1c_2  & \frac{\alpha c_2}{w_2}-c_2^2 & \cdots &-c_2c_n \\
\vdots &\vdots & \ddots & \vdots\\
-c_nc_1  & -c_nc_2 & \cdots & \frac{\alpha c_n}{w_n} -c_n^2
\end{bmatrix}=\DD-\sf cc^\top,
\end{equation} 
where $\alpha=\c^\top \w$ and $\DD=\operatorname{diag}( \frac{\alpha c_1}{w_1} ,\cdots, \frac{\alpha c_n}{w_n})>0 $. 
Note that $\widehat \LL(\c,\e)$ is the matrix associated with the combinatorial Laplacian of the recoverable complete network $\Gamma$.

\begin{propo}
    Let $\sf A\in \M_3(\RR)$ be a symmetric $M$-matrix, whose off-diagonal entries are negative. Then, $\sf A$ can be written as a rank-one perturbation of a positive diagonal matrix.
\end{propo}
\proof Let $\AA=(a_{ij})\in \M_3(\RR)$ be a symmetric $M$-matrix with negative off-diagonal entries so that, $a_{12}, a_{23},a_{13}<0$.
Define \[
c_1=\sqrt{\frac{(-a_{12})(-a_{13})}{-a_{23}}},\quad
c_2=\sqrt{\frac{(-a_{12})(-a_{23})}{-a_{13}}},\quad
c_3=\sqrt{\frac{(-a_{13})(-a_{23})}{-a_{12}}}.
\] 
Clearly, each $c_i$ is well-defined and positive, and the vector $\c=(c_1,c_2,c_3)^\top\in \RR^3$ satisfies $a_{ij}=-c_ic_j$, $i,j=1,2,3,$ and $i\ne j.$
Moreover, $\c$ is unique up to multiplication by $\pm 1$.
It follows that, $$\sf A=D_1-cc^\top,$$ where ${\sf D_1}=\operatorname{diag}(d_1,d_2,d_3)$ with $d_i=a_{ii}+c_i^2.$ 
Since all the entries of $\AA$ are nonzero, it is irreducible. Furthermore, as $\sf A$ is an $M$-matrix, $a_{ii}>0$ for $i=1,2,3$ and so, $d_i>0$, for each $i$.

Moreover, by the matrix determinant lemma, $\AA$ is singular if and only if $\sf \c^\top D_1^{-1}\c=1$ (which is equivalent to $\sum\limits_{i=1}^3\frac{c_i^2}{c_i^2+a_{ii}}=1$). In this case, setting $\sf x=D_1^{-1}c>0$, we obtain $$\sf Ax=(D_1-cc^\top)D_1^{-1}c=c-c(c^\top D_1^{-1}c)=0.$$
\qed

It is trivial to observe that $\widehat\LL(\sf c,w)$ is a rank one perturbation of a positive diagonal matrix. In Lemma \ref{groupinv}, we derive a closed-form representation of the group inverse of $\sf D-cc^\top$. To proceed, we recall the following lemma, where we modify the original result for the symmetric case so that the Moore-Penrose inverse coincides with the group inverse. An important fact is that the group inverse of a matrix $\AA$ exists if and only if $\operatorname{rank}(\AA)=\operatorname{rank}(\AA^2)$. Therefore, for a symmetric matrix $\AA$, the group inverse $\AA^\#$ exists; see \cite{BIG03}.

\begin{lemma}\cite[Theorem 2.1]{BT2003} \label{bakres}
For a symmetric matrix $\sf A\in \M_n(\RR)$ and $\sf \b \in \mathbb{R}^n$, we set $\MM:=\sf A-bb^\top.$ Let $\sf f:=A^{\#}b$ be a nonzero vector and $\sf \delta:={\parallel f\parallel}^2$.
Then
$$\sf M^{\#}=A^{\#}-\frac{1}{\delta}ff^\top A^{\#}-\frac{1}{\delta}A^{\#}ff^\top+\frac{f^\top A^{\#}f}{\delta^2}ff^\top.$$
\end{lemma}
\begin{lemma}\label{groupinv}
        The entries of the group inverse of the matrix $\DD-\sf cc^\top$ or the Green function $G_\omega(x_i,x_j)$ are given by
    \begin{equation*}
        (\DD-\c\c^\top)^\#_{ij}=G_\omega(x_i,x_j)=
        \begin{cases}
            \dfrac{w_iw_j}{\alpha ||\w||^2}\left(\dfrac{1}{||\w||^2}\sum\limits_{k=1}^n\displaystyle\dfrac{w_k^3}{c_k}-\dfrac{w_i}{c_i}-\dfrac{w_j}{c_j}\right), & i\neq j\\
            \displaystyle\dfrac{w_i}{\alpha c_i}+\dfrac{w_i^2}{\alpha ||\w||^2}\left(\dfrac{1}{||\w||^2}\sum\limits_{k=1}^n\dfrac{w_k^3}{c_k}-\dfrac{2w_i}{c_i}\right), & i= j.
            \end{cases}
    \end{equation*}
\end{lemma}
\proof
By Lemma \ref{bakres}, we get $${\sf (\DD- cc^\top)^\#=\DD^{-1}-\frac{ ww^\top \DD^{-1}}{||\w||^2}-\frac{\DD^{-1} ww^\top}{||\w||^2}}+
\frac{\sf w^\top D^{-1}w}{ ||\w||^4}\w \w ^\top.$$
Using the fact that $ \DD^{-1}\w=\left[\frac{w_1^2}{\alpha c_1}, \ldots,\frac{w_n^2}{\alpha c_n}\right]^\top$, the result is immediate.\qed

We remark that the above result was obtained in  \cite[Theorem 8.6]{CEM2019} using discrete potential theory tools and using the electrical  equivalence between a star network and a recoverable complete network. We have performed the proof here using purely algebraic tools as an alternative. 

We are now ready to determine when the group inverse of $\widehat \LL(\c,\w)$ is an $M$-matrix.
When $n=2$, $\widehat \LL(\c,\w)^\# $ is an $M$-matrix for any $\c,\w>0$; see Theorem \ref{complete:inverse}. However, for $n\ge 3$, there exist complete recoverable networks that do not have the $M$-property; that is, for such networks, $\widehat \LL(\c,\w)^\# $ fails to be an $M$-matrix. The following example illustrates this fact.
\begin{ex}
   Let $\w=(1,2,1)^\top$ and ${\sf c}=(2,1,2)^\top$. Then, the matrix associated with $\L_{q_{\omega}}$ on a  finite recoverable complete network: $$\widehat \LL(\c,\w)=\begin{bmatrix}
       8 & -2 & -4\\
       -2 & 2 & -2 \\
       -4 & -2 & 8
   \end{bmatrix}.$$
   But, the group inverse $\widehat \LL(\c,\w)^\#=\frac{1}{72}\begin{bmatrix}
       7 & -4 & 1\\
       -4 & 4 & -4\\
       1 & -4 & 7
   \end{bmatrix}$ has a positive off-diagonal entry, and hence $\widehat \LL(\c,\w)^{\#}$ is not a $Z$-matrix.
\end{ex}
The following theorem establishes the necessary and sufficient condition under which $\widehat\LL(\sf c,w)^\#$ qualifies as an $M$-matrix. 
\begin{theorem}
  \label{complete:inverse}
Given $\c,\w>0$, the matrix $\widehat \LL(\c,\w)^\# $ is an $M$-matrix iff
$$\sum\limits_{k=1}^n\dfrac{w_k^3}{c_k}\le||\w||^2\Big( \dfrac{w_i}{c_i}+\dfrac{w_j}{c_j}\Big), \hspace{.25cm}\hbox{for any \,\,  $i,j=1,\ldots,n,\,\,i\not=j$}.$$
In particular,  when $n=2$ then $\widehat \LL(\c,\w)^\# $ is an $M$-matrix for any $\c,\w>0$.
\end{theorem}

\proof  
As the matrix $\widehat \LL(\c,\w)$ is positive semidefinite, so is its group inverse. Therefore, we require $\widehat \LL(\c,\w)^\#$ to be an $M$-matrix precisely when all its off-diagonal entries are nonpositive. By Lemma \ref{groupinv}, $\widehat \LL(\c,\w)^\#$ is an $Z$-matrix iff the stated inequalities follow. Hence the result.\qed

\vspace{.2cm}

From the above characterization, we easily obtain singular, positive semidefinite Schr\"odinger operators on a finite recoverable complete network that satisfy the $M$-property with respect to $\omega$.

\begin{rem}
For $\c,\w>0$, $\widehat\LL(\c,\w)^\#$ is an $M$-matrix iff $\widehat\LL(t\c,s\w)^\#$ is an $M$-matrix for any $t,s>0.$
Also, note that $\widehat\LL(\c,s\w)=\widehat\LL(\c,\w)$ for any $s>0.$
\end{rem}
\begin{corollary}\label{posmultiple}
    If $\c$ and $\w$ are positive multiples of each other, then $\widehat \LL(\c,\w)^\#$ is an $M$-matrix.
\end{corollary}

In particular, when $\sf w=c>0$, $\widehat \LL(\sf c,c)=||\c||^2I-\sf cc^\top$ and $$\widehat \LL(\sf c,c)^\#=\frac{1}{||\c||^4}(||\c||^2I-\sf cc^\top),$$ which is a singular, symmetric and irreducible $M$-matrix.\\

The following corollary provides the necessary and sufficient conditions for  the group inverse of the combinatorial Laplacian of the recoverable complete network $\Gamma=(V,\widehat c)$ to be an $M$-matrix.
 
\begin{corollary}
\label{complete-M-c}
Given $\c>0$, the matrix $\widehat \LL(\c,\e)^\#$ is an $M$-matrix iff 
$$\sum\limits_{k=1}^n\dfrac{1}{c_k}\le n \Big( \dfrac{1}{c_i}+\dfrac{1}{c_{j}}\Big),\hspace{.15cm}i,j=1,\ldots,n,\,\,i\not=j.$$
In particular, the above inequalities always hold when $c_1=\cdots=c_n$. 
\end{corollary}
\begin{rem}
   When $c_1=\cdots=c_n=1$, the combinatorial Laplacian $\sf \widehat L(c,e)$ reduces to the standard graph Laplacian. From the preceding corollary, it follows that the group inverse of the graph Laplacian of a finite, simple, connected complete graph is an $M$-matrix. 
\end{rem}
We observe that there exists $\c>0$ for which $\Gamma$ does not have $M$-property or the group inverse of the matrix associated to the combinatorial Laplacian is not an $M$-matrix.

Consider $n\ge 3$, $\c>0$ such that $c_i=c_j>2(n-1)c_m$ for some $i,j,m=1,\ldots,n$, where  $i\not=j$, $i\not= m$ and $m\not=j$. Then
$$\sum\limits_{k=1}^n\dfrac{1}{c_k}\ge \dfrac{1}{c_m}+\dfrac{2}{c_i}>\dfrac{2 n}{c_i}= n\Big( \dfrac{1}{c_i}+\dfrac{1}{c_{j}}\Big).$$
Thus, such $\sf c>0$ does not belong to the class of vectors for which $\sf \widehat L(c,e)^\#$ is an $M$-matrix.

The following result establishes a sufficient condition for $\widehat \LL(\c,\sf e)^\#$ to be an $M$-matrix.
\begin{corollary}
    Let $p=\max\limits_{i=1,\ldots,n}\{c_i\}$ and $q=\min\limits_{i=1,\ldots,n}\{c_i\}$. If $p\leq 2q$, then $\widehat \LL(\c,\sf e)^\#$ is an $M$-matrix.
\end{corollary}
\proof We have $$\sum\limits_{k=1}^n\dfrac{1}{c_k}\le \dfrac{n}{q}\le \dfrac{2n}{p} \le n\Big( \dfrac{1}{c_i}+\dfrac{1}{c_{j}}\Big),\; i,j=1,\ldots,n, \; i\neq j.$$ Thus, by Corollary \ref{complete-M-c}, the result holds.\qed
\vspace{.2cm}

\section{Families of complete networks satisfying the $M$-property} 

In the sequel, we assume that $n\ge 3$, since for $n=2$, $\sf \widehat L(c,w)^\#$ is an $M$-matrix for every $\sf c,w>0$.
Now, we reformulate the above expression for $\widehat \LL(\c,\w )$ by defining $$ d_k=\dfrac{ c_k}{w_k},\; k=1,\ldots,n.$$
Then, $\alpha=\displaystyle \sum\limits_{k=1}^n\dfrac{c_k^2}{d_k}$ and 
$$\widehat \LL(\c,\w)=\alpha \DD_\d-\c\c^\top=\begin{bmatrix}  \alpha d_1-c_1^2 & -c_1c_2 & \cdots & -c_1c_n\\
-c_1c_2  & \alpha d_2-c_2^2 & \cdots &-c_2c_n \\
\vdots &\vdots & \ddots & \vdots\\
-c_nc_1  & -c_nc_2 & \cdots & \alpha d_n -c_n^2
\end{bmatrix}, $$
 where ${\sf d}=(d_1,\ldots,d_n)^\top \in \RR^n$ and $\DD_\d$ is the diagonal matrix whose diagonal entries are given by vector $\d$.

Observe that, if $\d=\c$, then $\alpha\DD_\d-\c\c^\top$ is the combinatorial Laplacian of the recoverable complete network $\Gamma$. From now on, given a vector $\d>0$, we denote by $d_*=\min\limits_{k=1,\ldots,n}\{d_k\}$, $d^*=\max\limits_{k=1,\ldots,n}\{d_k\}$, and $$d^{**}=
    \left\{\begin{array}{cl}
       d^*, & \text{if $d^*$ is attained at more than one point.}\\
        \max\limits_{k=1,\ldots,n}\{d_k:d_k<d^*\}, & \text{otherwise.}
    \end{array}\right.$$

In the following result, we reformulate the previous findings in terms of $\c$ and $\d$ in order to characterize families of complete networks satisfying the $M$-property. 

\begin{theorem}\label{complete:c}
All singular, symmetric and irreducible $M$-matrices supported by a complete recoverable network  on $n$ vertices are given by the expression $\alpha\DD_\d-\c\c^\top$,
where {$ \sf c,d>0$}  and $\alpha=\displaystyle \sum\limits_{k=1}^n\dfrac{c_k^2}{d_k}.$
Moreover, $\big(\alpha\DD_\d-\c\c^\top\big)^\#$ is  an $M$-matrix iff 
$$\displaystyle \sum\limits_{k=1}^n\dfrac{c_k^2}{d_k^3}\le \Big(\dfrac{1}{d_i}+\dfrac{1}{d_j}\Big)\Big(\sum\limits_{k=1}^n\dfrac{c_k^2}{d_k^2}\Big),\hspace{.15cm}i,j=1,\ldots,n,\,\,i\not=j,$$ 
or equivalently iff $$\displaystyle \sum\limits_{k=1}^n\dfrac{c_k^2}{d_k^3}\le \Big(\dfrac{1}{d^*}+\dfrac{1}{d^{**}}\Big)\Big(\sum\limits_{k=1}^n\dfrac{c_k^2}{d_k^2}\Big).$$ 
\end{theorem}
\proof 
If $\AA$ is a singular, symmetric and irreducible $M$-matrices supported by an $n$-complete recoverable network, then there exist $\c, \w>0$ such that $\AA=\widehat \LL(\c,\w)$ (see \eqref{Schr2}). Choosing $d_k=\dfrac{ c_k}{ w_k}$, for any $k=1,\ldots,n$, then $\widehat\LL(\c,\w)=\alpha\DD_\d-\c\c^\top,$ where $\alpha=\displaystyle \sum\limits_{k=1}^n\dfrac{c_k^2}{d_k}.$

By Theorem \ref{complete:inverse}, we conclude that $\Big(\alpha\DD_\d-\c\c^\top\Big)^\#$ is an $M$-matrix iff 
$$\sum\limits_{k=1}^n\dfrac{c_k^2}{d_k^3}=\sum\limits_{k=1}^n\dfrac{w_k^3}{c_k}\le||\w||^2\Big( \dfrac{w_i}{c_i}+\dfrac{w_j}{c_j}\Big)=\Big( \dfrac{1}{d_i}+\dfrac{1}{d_j}\Big)\Big(\sum\limits_{k=1}^n\dfrac{c_k^2}{d_k^2}\Big), \hspace{.25cm}\hbox{for any \,\,  $i,j=1,\ldots,n,\,\,i\not=j$},$$
or equivalently iff $$\Big(\sum\limits_{k=1}^n\dfrac{c_k^2}{d_k^3}\Big)\leq \min\limits_{i
}\Big\{\frac{1}{d_i}+\min\limits_{k\neq i}\Big(\frac{1}{d_k}\Big)\Big\}\Big(\sum\limits_{k=1}^n\dfrac{c_k^2}{d_k^2}\Big)=\Big(\frac{1}{d^*}+\frac{1}{d^{**}}\Big)\Big(\sum\limits_{k=1}^n\dfrac{c_k^2}{d_k^2}\Big),$$ and thus, the final inequality follows.
\qed

\begin{corollary}\label{equald}
Let $\c>0$ be arbitrary. If all the $d_k$ are equal, then $(\alpha \sf D_d-cc^T)^\#$ is an $M$-matrix.
\end{corollary}

Note that if $\d, \c$ satisfy the inequalities in Theorem \ref{complete:c}, then for any $t,s>0$, the scaled vectors $\widehat \d=t\d,\widehat \c= s\c$ also satisfy those inequalities. 
In fact, $\widehat \alpha\DD_{\hat \d}-\widehat\c\widehat\c^\top=s^2(\alpha\DD_\d-\c\c^\top)$, where $\widehat\alpha=\frac{s^2}{t}\alpha$.
It is well known that for an $n\times n$ real, singular matrix $\sf A$ (when the group inverse exists) and for $r\ne 0$, $(r {\sf A})^\#=\frac{1}{r}\sf A^\#.$ Thus, for a positive scalar $r$, the $M$-matrix property is preserved by the group inverse.
For this reason, we obtain the following corollary. 

\begin{corollary}
    Let $\widehat \d=t\d$ and $\widehat\c=s\c$ for any $t,s>0$. Then, $(\widehat \alpha\DD_{\widehat\d}-\widehat\c\widehat\c^\top)^\#$ is an $M$-matrix if and only if $(\alpha\DD_\d-\c\c^\top)^\#$ is an $M$-matrix.
\end{corollary}

We can extend Corollary \ref{equald} by allowing two different diagonal values. To do this, we need the following lemma.

\begin{lemma}
\label{polynomial} Given $b>0$, the unique real root of the polynomial $p(x)=x^3+b(2x-1)$ is  
$$r(b)=\dfrac{\sqrt[3]{12b}}{6\sqrt[3]{9+ \sqrt{96b + 81}}}\,\Big(\sqrt[3]{(9+ \sqrt{96b + 81})^2}-2\sqrt[3]{12b}\Big).$$
Moreover, $r(b)\in \big(0,\frac{1}{2}\big)$. 
\end{lemma}
\begin{corollary}
\label{weight:c}
Let $d,s>0$ and $\emptyset \not=I\subsetneq \{1,\ldots,n\}$.
Given $\c>0$, define the vector $\d$ as 
$$d_k=\begin{cases}
    s,\;  k\in I\\
    d, \; k \notin I.
\end{cases}$$
If $k_s=\displaystyle \sum\limits_{k\in I}c_k^2$ and $k_d=||\c||^2-k_s$, then $(\alpha\DD_{\d}-\c\c^\top)^\#$ is an $M$-matrix iff one of the following conditions hold:
\begin{itemize}
\item[{\rm (i)}] $|I|=1$ and $d\le s$. 
\item[{\rm (ii)}] $|I|= n-1$ and $s\le d$.
\item[{\rm (iii)}] $2\le |I|\le n-2$ and $r\Big(\dfrac{k_d}{k_s}\Big) \le \dfrac{d}{s}\le \Big(r\Big(\dfrac{k_s}{k_d}\Big)\Big)^{-1}.$  In particular, this condition holds when $\dfrac{1}{2}\le \dfrac{d}{s}\le 2$.
\end{itemize}

\end{corollary}
\proof First observe that 
$$\displaystyle \sum\limits_{k=1}^n\dfrac{c_k^2}{d_k^3}=\dfrac{k_s}{s^3}+\dfrac{k_d}{d^3},\hspace{.5cm}\sum\limits_{k=1}^n\dfrac{c_k^2}{d_k^2}=\dfrac{k_s}{s^2}+\dfrac{k_d}{d^2},\hspace{.5cm}\frac{1}{d^*}+\dfrac{1}{d^{**}}=\left\{\begin{array}{cl}
\dfrac{2}{s}, & d\le s, |I|\ge 2,\\[3ex]
\dfrac{2}{d}, & s\le d, |I|\le n-2,\\[3ex]
\dfrac{1}{s}+\dfrac{1}{d}, & \hbox{otherwise}.\end{array}\right.$$ 
\noindent\textbf{Case 1.} {$d\le s$ and $|I|\ge 2$}\\

By Theorem \ref{complete:c}, $(\alpha\DD_{ \d}-\c\c^\top)^\#$ is an $M$-matrix iff
$$\dfrac{k_s}{s^3}+\dfrac{k_d}{d^3}\le \dfrac{2}{s}\Big(\dfrac{k_s}{s^2}+\dfrac{k_d}{d^2}\Big),$$
iff $s^2(s-2d)k_d\le d^3k_s$, or equivalently, iff $t^3+ b(2t-1)\ge 0$, where $t=\dfrac{d}{s}$ and $b=\dfrac{k_d}{k_s}$. 
Now, using Lemma \ref{polynomial}, $r(b)$ is the unique real root of the equation $t^3+b(2t-1)=0$. Hence, $(\alpha \DD_\d-\c\c^\top)^\#$ is an $M$-matrix iff $$t\ge r(b),\; \text{i.e.}\; \frac{d}{s}\ge r\Big(\frac{k_d}{k_s}\Big).$$

\noindent\textbf{Case 2.}
$s\le d$ and $|I|\le n-2$\\

Again by Theorem \ref{complete:c}, $(\alpha\DD_{ \d}-\c\c^\top)^\#$ is an $M$-matrix iff
$$\dfrac{k_s}{s^3}+\dfrac{k_d}{d^3}\le \dfrac{2}{d}\Big(\dfrac{k_s}{s^2}+\dfrac{k_d}{d^2}\Big),$$
iff $d^2(d-2s)k_s\le s^3k_d$, or equivalently, iff $\Tilde{t}^3+ \Tilde{b}(2\Tilde{t}-1)\ge 0$, where $\Tilde{t}=\dfrac{s}{d}$ and $\Tilde{b}=\dfrac{k_s}{k_d}$. Again, by Lemma \ref{polynomial}, we have $\dfrac{s}{d}\ge r\Big(\dfrac{k_s}{k_d}\Big) $.

Combining case 1 and case 2, we get the condition (iii).

\noindent\textbf{Case 3.} {Either $d\le s$ and $|I|=1$ or $s\le d$ and $|I|=n-1$}\\

Note that,
$$\dfrac{k_s}{s^3}+\dfrac{k_d}{d^3}<\dfrac{k_s}{s^3}+\frac{k_d}{sd^2}+\dfrac{k_d}{d^3}+\frac{k_s}{s^2d}
=\Big(\dfrac{1}{s}+\dfrac{1}{d}\Big)\Big(\dfrac{k_s}{s^2}+\dfrac{k_d}{d^2}\Big).$$
Thus, the inequalities in Theorem \ref{complete:c} are satisfied and, in this case,
%
$(\alpha\DD_{\d}-\c\c^\top)^\#$ is an $M$-matrix. 
\qed

\begin{corollary}
\label{weight:c}
Let $\d>0$ such that $d^*\le \sqrt[3]{2}d_*$. Then, for any $\c>0$, $(\alpha\DD_{\d}-\c\c^\top)^\#$ is an $M$-matrix.
\end{corollary}

\proof We have  that 
$$\displaystyle \sum\limits_{k=1}^n\dfrac{c_k^2}{d_k^3}\le \dfrac{||\c||^2}{d_*^3}\le \dfrac{2||\c||^2}{(d^*)^3}\le \Big( \dfrac{1}{d_i}+\dfrac{1}{d_j}\Big)\Big(\sum\limits_{k=1}^n\dfrac{c_k^2}{d_k^2}\Big),\hspace{.15cm}i,j=1,\ldots,n,\,\,i\not=j.$$ 
Hence, the conclusion holds.
\qed

The following corollary is used in the proof of Proposition \ref{complete-c-d}.
\begin{corollary}\label{sufficient}
   Given $\c>0$, if $\displaystyle\sum\limits_{k=1}^n\dfrac{c_k^2}{d_k^3}\le \dfrac{2}{d^*}\Big(\sum\limits_{k=1}^n\dfrac{c_k^2}{d_k^2}\Big)$, then $(\alpha \DD_\d-\c\c^\top)^\#$ is an $M$-matrix.
\end{corollary}

Now, we show how to construct singular, symmetric, and irreducible $M$-matrices on recoverable complete networks such that their group inverse is also an $M$-matrix by accurately choosing the conductances and diagonal values of $\DD_\d$. We can assume w.l.o.g. that the diagonal values are ordered, specifically $0<d_1\le \cdots\le d_n$.

\begin{propo}
\label{complete-c-d}
Given $c_1>0,\; d_1>0$, if $c_k>0 \; (k=2,\ldots,n)$ are chosen arbitrarily and $d_k>0\; (k=2,\ldots,n)$ are chosen recursively satisfying the inequalities
$$0<d_{k-1}\le d_k\le 2\Big(\sum\limits_{j=1}^{k-1}\dfrac{c_j^2}{d_j^2}\Big)\Big(\sum\limits_{j=1}^{k-1}\frac{c_j^2}{d_j^3}\Big)^{-1},$$ then $(\alpha\DD_{ \d}-\c\c^\top)^\#$ is an $M$-matrix.
\end{propo}
\proof 
Let $\c>0$. Given $d_1>0$, we show that $d_k(k=2,\ldots,n)$ can be chosen such that \begin{equation}\label{induct}
d_{k-1}\le d_k\le 2\Big(\sum\limits_{j=1}^{k-1}\frac{c_j^2}{d_j^2}\Big)\Big(\sum\limits_{j=1}^{k-1}\frac{c_j^2}{d_j^3}\Big)^{-1}.
\end{equation}
We apply induction on $k$. Since $d_1<2d_1=2\Big(\dfrac{c_1^2}{d_1^2}\Big)\Big(\dfrac{c_1^2}{d_1^3}\Big)^{-1}$, we choose $d_2$ such that $$d_1\le d_2 \le 2\Big(\dfrac{c_1^2}{d_1^2}\Big)\Big(\dfrac{c_1^2}{d_1^3}\Big)^{-1}.$$ 
Thus, the base step ($k=2$) is established. Assume that, the inequality (\ref{induct}) is true for $k=m$. Then, we have 
$$d_{m-1}\le d_m\le 2\Big(\sum\limits_{j=1}^{m-1}\frac{c_j^2}{d_j^2}\Big)\Big(\sum\limits_{j=1}^{m-1}\frac{c_j^2}{d_j^3}\Big)^{-1}.$$
Now $$\sum\limits_{j=1}^{m}\frac{c_j^2}{d_j^3}=\dfrac{c_m^2}{d_m^3}+\sum\limits_{j=1}^{m-1}\frac{c_j^2}{d_j^3}< \dfrac{2c_m^2}{d_m^3}+\dfrac{2}{d_{m}}\Big(\sum\limits_{j=1}^{m-1}\dfrac{c_j^2}{d_j^2}\Big)=\dfrac{2}{d_m}\sum\limits_{j=1}^{m}\dfrac{c_j^2}{d_j^2}.$$
So, we get $$d_m<2 \Big(\sum\limits_{j=1}^{m}\dfrac{c_j^2}{d_j^2}\Big)\Big(\sum\limits_{j=1}^{m}\dfrac{c_j^2}{d_j^3}\Big)^{-1}.$$
By choosing $d_{m+1}$ such that
$$d_m\le d_{m+1}\le 2 \Big(\sum\limits_{j=1}^{m}\dfrac{c_j^2}{d_j^2}\Big)\Big(\sum\limits_{j=1}^{m}\dfrac{c_j^2}{d_j^3}\Big)^{-1},$$ we obtain the inequality (\ref{induct}).
By setting $k=n+1$ in the inequality, we have
$$d_{n}\le 2 \Big(\sum\limits_{j=1}^{n}\dfrac{c_j^2}{d_j^2}\Big)\Big(\sum\limits_{j=1}^{n}\dfrac{c_j^2}{d_j^3}\Big)^{-1},$$ and since $d_n=d^*$,
 Corollary \ref{sufficient} applies.
\qed

An illustration of Proposition 
\ref{complete-c-d} is presented next.

\begin{ex}
Here, we construct a $3\times 3$ singular, irreducible, symmetric $M$-matrix arising from $\Gamma$ whose group inverse is an $M$-matrix.
   Let $c_1=1, d_1=2.$ Now, $2\Big(\dfrac{c_1^2}{d_1^2}\Big)\Big(\dfrac{c_1^2}{d_1^3}\Big)^{-1}=4.$ Choose $d_2=3$ lying between $d_1=2$ and $2\Big(\dfrac{c_1^2}{d_1^2}\Big)\Big(\dfrac{c_1^2}{d_1^3}\Big)^{-1}=4$ and $c_2=1$, chosen arbitrarily. Now
   $$2\Big(\frac{c_1^2}{d_1^2}+\frac{c_2^2}{d_2^2}\Big)\Big(\frac{c_1^2}{d_1^3}+\frac{c_2^2}{d_2^3}\Big)^{-1}=\frac{156}{35}.$$ Next, we choose $d_3=4$ and $c_3=2$.
   So, we have:
   $$\alpha\DD_\d-\c\c^\top=\begin{bmatrix}
       \frac{8}{3} & -1 & -2\\
       -1 & \frac{9}{2} & -2\\
       -2 & -2 & \frac{10}{3}
   \end{bmatrix},$$ and its group inverse is given by: $$\frac{1}{10648}\begin{bmatrix}
       1401 & -738 & -909 \\
-738 & 1620 & -342 \\
-909 & -342 & 1137
       \end{bmatrix},$$
       which is an $M$-matrix since it is a $Z$-matrix and all its principal minors are nonnegative.
   \end{ex}
   
\begin{rem}
    The recursive inequalities in Proposition \ref{complete-c-d} are not necessary for $(\alpha \DD_\d-\c\c^\top)^\#$ to be an $M$-matrix.
Let $\c=(1,1,1)^\top$, $d_1=d_2=2$ and $d_3=5$. Now $$2\Big(\frac{c_1^2}{d_1^2}+\frac{c_2^2}{d_2^2}\Big)\Big(\frac{c_1^2}{d_1^3}+\frac{c_2^2}{d_2^3}\Big)^{-1}=4,$$ and $d_3\notin[2,4]$. 
Here,
$\alpha\DD_\d-\c\c^\top=\begin{bmatrix}
    \frac{7}{5} & -1 & -1\\
    -1 & \frac{7}{5} & -1\\
    -1 & -1 & 5
\end{bmatrix}$ and its group inverse: $$(\alpha\DD_\d-\c\c^\top)^\#=\frac{1}{5832}\begin{bmatrix}
    1255 & -1175 & -200\\
    -1175 & 1255 & -200\\
    -200 & -200 & 1000
\end{bmatrix},$$ which is an $M$-matrix since it is a $Z$-matrix and all its principal minors are nonnegative.

\end{rem}

Next, we obtain the conditions for the group inverse of the combinatorial Laplacian of a recoverable complete network to be an $M$-matrix, or in an equivalent manner, for a recoverable complete network to have the $M$-property. For this, it suffices to take $d_j=c_j$ for  $j=1,\ldots,n$ in Proposition \ref{complete-c-d}. 

\begin{corollary}
\label{star-M-constant}
Given  $0<c_1\le \cdots\le c_n$,  $(\alpha\DD_{ \c}-\c\c^\top)^\#$ is an $M$-matrix if for any $k=2,\ldots,n$, the value $c_k$  satisfies the following inequalities
$$c_{k-1}\le  c_k \le 2(k-1)\Big(\sum\limits_{j=1}^{k-1}\frac{1}{c_j}\Big)^{-1}.$$
In particular, the above inequalities always hold when $c_1=\cdots=c_n$. 
\end{corollary}

\begin{corollary}
\label{star-M}
Given $\beta>0$ and $0<d_1\le \cdots \le d_n$, the matrix
$$\begin{bmatrix}  \alpha d_1-\beta^2 & -\beta^2 & \cdots & -\beta^2\\
-\beta^2  & \alpha d_2-\beta^2 & \cdots &-\beta^2 \\
\vdots &\vdots & \ddots & \vdots\\
-\beta^2  & -\beta^2 & \cdots & \alpha d_n -\beta^2
\end{bmatrix}^\#,$$ where $\alpha=\beta^2\sum\limits_{k=1}^n\frac{1}{d_k}$,
is  an $M$-matrix when the values $d_k$, $k=2,\ldots,n$  are chosen recursively satisfying the following inequalities
$$0<d_{k-1}\le d_k\le 2\Big(\sum\limits_{j=1}^{k-1}\dfrac{1}{d_j^2}\Big)\Big(\sum\limits_{j=1}^{k-1}\frac{1}{d_j^3}\Big)^{-1},\hspace{.25cm} k=2,\ldots,n.$$
\end{corollary}

Note that the choice $d_j=\beta,\; j=1,\ldots,n$ in the above corollary recovers the case of the weighted Laplacian with a positive constant weight $\beta$  for a complete graph. Hence, we newly show that a complete graph has the $M$-property.
\vspace{.2cm}

\section{Preservation of the $M$-property under network equivalence}\label{complete-star}

We conclude this work by analyzing the behavior of the $M$-property when considering a recoverable complete network on $n$ vertices and its equivalent star on $n+1$ vertices.
For this purpose, we recall a result (see Theorem \ref{starm}) from \cite{CEPS2026}, where the authors investigated the $M$-property for star networks.

Let $\c=(c_1,\ldots,c_n)^\top >\sf 0$ and $\d=(d_1,\ldots,d_n)^\top >0$ be as defined above, and set $\alpha=\displaystyle \sum\limits_{k=1}^n\dfrac{c_k^2}{d_k}$. Consider the matrix
$$\sf M(\c,\d)=\begin{bmatrix}  \alpha& -c_1 & \cdots & -c_n\\
-c_1 & d_1& \cdots & 0\\
\vdots &\vdots & \ddots & \vdots\\
-c_n & 0 & \cdots & d_n
\end{bmatrix}.$$ 
As $\c$ and $\d$ range over all positive vectors, the family $\sf M(\c,\d)$ represents all singular, symmetric, and irreducible $M$-matrices arising from an $(n+1)$-star network with conductances given by $c(x_0,x_i)=c_i$, for $i=1,\ldots,n$. Here, $x_0$ denotes the center of the star, and $x_1,\ldots,x_n$ are its adjacent vertices; see \cite{CEPS2026}.

Alternatively, it is immediate that $\sf M(\c,\d)$ is irreducible (since the underlying star is connected), symmetric (by construction), and a $Z$-matrix (as its off-diagonal entries are nonpositive). Consider the positive vector ${\sf x}=\Big[1,\frac{c_1}{d_1},\ldots,\frac{c_n}{d_n}\Big]^\top$. A direct computation shows that $\MM(\c,\d)\sf x=0$, and hence ${\sf x}$ belongs to the kernel of $\sf M(\c,\d)$. To prove that $\sf M(\c,\d)$ is an $M$-matrix, we show that, for any $\epsilon>0$, the matrix $\MM(\c,\d)+\epsilon \II$ is an invertible $M$-matrix. Indeed, since $(\MM(\c,\d)+\epsilon \II){\sf x}=\epsilon {\sf x}>{\sf 0}$, the matrix $\MM(\c,\d)+\epsilon \II$ is semipositive. Therefore, $\MM(\c,\d)+\epsilon \II$ is an invertible $M$-matrix, and consequently, $\sf M(\c,\d)$ is an $M$-matrix. In particular, $\MM(\c,\d)$ is a singular, irreducible, symmetric $M$-matrix arising from a star network. Following the results in \cite{CEM2019}, after applying a star–complete transformation, the matrix of the corresponding Schr\"odinger operator on the complete network is given by $\alpha \DD_\d-\c\c^\top$. From an electrical  point of view, both networks are equivalent, since the effective resistance on the complete network coincides with the restriction of the effective resistance of the star; see \cite[Corollary 7.8]{CEM2019}. Moreover, as mentioned above, $(\alpha \DD_\d-\c\c^\top)^\#$ can be obtained from $\MM(\c,\d)^\#$ and vice versa. The question we address in this section is whether this electrical equivalence also implies equivalence of the $M$-property; that is, whether for given $\c,\d>0$, the matrix $(\alpha \DD_\d-\c\c^\top)^\#$ is an $M$-matrix if and only if $\MM(\c,\d)^\#$ is also an $M$-matrix. We begin by recalling the characterization of the $M$-property for star networks.

\begin{theorem}\cite[Theorem 4.6]{CEPS2026}\label{starm}
Given $\c,\d>0$,  $\sf M(\c,\d)^\#$ is an $M$-matrix iff 
$$\sum_{k=1}^n\frac{c^2_k}{d^3_k}\leq \frac{1}{d^*}\Big(1+\sum\limits_{k=1}^n\frac{c^2_k}{d^2_k}\Big).$$
\end{theorem}

Next, we show that, in general, the answer to the previously raised question is negative in both directions. In other words, there exist $\c,\d> 0$ such that ${\MM}(\c,\d)^\#$ is an $M$-matrix, whereas $(\alpha \DD_\d-\c\c^\top)^\#$ is not an $M$-matrix, and there also exist $\c,\d>\sf 0$ such that $(\alpha \DD_\d-\c\c^\top)^\#$ is an $M$-matrix, whereas $\sf M(\c,\d)^\#$ is not an $M$-matrix. Therefore, the $M$-property cannot be fully characterized in purely electrical terms.

\begin{ex}
 Let $\c=(1,1,1)^\top$ and $\d=(2,5,5)^\top$, which imply that $\alpha=\frac{9}{10}$ and also that 
 $$\MM(\c,\d)=\begin{bmatrix}
     \frac{9}{10} & -1 & -1 & -1\\
     -1 & 2 & 0& 0\\
     -1 & 0 & 5 & 0\\
     -1 & 0 & 0 & 5
 \end{bmatrix}, \quad \textit{and}  \hspace{0.4cm} \alpha\DD_{ \d}-\c\c^\top=\begin{bmatrix}
     \frac{4}{5} & -1 & -1\\
     -1 & \frac{7}{2} & -1\\
     -1 & -1 & \frac{7}{2}
 \end{bmatrix}.$$ 
 Then, $\MM(\c,\d)^\#$ is an $M$-matrix as the inequality $\displaystyle \sum\limits_{k=1}^3\frac{c_k^2}{d_k^3}=\frac{141}{1000}<\frac{1}{d^*}\Big(1+\sum\limits_{k=1}^3\frac{c_k^2}{d_k^2}\Big)=\frac{133}{500}$ is satisfied, but 
 $(\alpha \DD_\d-\c\c^\top)^\#$   is not an $M$-matrix as $$\sum\limits_{k=1}^3\frac{c_k^2}{d_k^3}=\frac{141}{1000}\nleq \frac{33}{250}= \Big(\frac{1}{d_2}+\frac{1}{d_3}\Big)\sum\limits_{k=1}^3\frac{c_k^2}{d_k^2}.$$
\end{ex}

\begin{ex}
Let $\c=(1,1,1)^\top$ and $\d=(1,2,1)^\top$, which imply that $\alpha=\frac{5}{2}$ and also that 
$$\MM(\c,\d)=\begin{bmatrix}
    \frac{5}{2} & -1 & -1 & -1\\
    -1 & 1 & 0 & 0\\
    -1 & 0 & 2 & 0\\
    -1 & 0 & 0 & 1
\end{bmatrix}, \quad \textit{and}  \hspace{0.4cm} \alpha\DD_\d-\c\c^\top=\begin{bmatrix}
    \frac{3}{2} & -1 & -1\\
    -1 & 4 & -1\\
    -1 & -1 & \frac{3}{2}
\end{bmatrix}.$$
Since $d^*=2$ and $d^{**}=1$, $(\alpha\DD_{ \d}-\c\c^\top)^\#$ is an $M$-matrix as the inequality 
$$  \sum\limits_{k=1}^3\dfrac{c_k^2}{d_k^3}=\dfrac{17}{8}\le \dfrac{27}{8}= \Big(\dfrac{1}{d^*}+\dfrac{1}{d^{**}}\Big)\Big(\sum\limits_{k=1}^3\dfrac{c_k^2}{d_k^2}\Big)$$ 
holds, but $\MM(\c,\d)^\#$ is not an $M$-matrix as
$\displaystyle \sum\limits_{k=1}^3\frac{c_k^2}{d_k^3}=\dfrac{17}{8}\nleq \dfrac{13}{8}=\frac{1}{d^*}\Big(1+\sum\limits_{k=1}^3\frac{c_k^2}{d_k^2}\Big).$ 
\end{ex}

In the following result, which directly follows from the characterization obtained above, we identify conditions on the parameters $\c$ and $\d$ under which the $M$-property is preserved between a star and its associated complete network.

\begin{propo}
\label{equivalence}
Given $\c,\d>0$, the following statements hold:
\begin{enumerate}[{\rm (i)}]
\item If $\displaystyle\sum\limits_{k=1}^n\frac{c_k^2}{d_k^2}> \dfrac{d^{**}}{d^*}$ and $\sf M(\c,\d)^\#$ is an $M$-matrix, then $(\alpha\DD_{\d}-\c\c^\top)^\#$ is an $M$-matrix.
\item If $\displaystyle\sum\limits_{k=1}^n\frac{c_k^2}{d_k^2}< \dfrac{d^{**}}{d^*}$ and $(\alpha\DD_{\d}-\c\c^\top)^\#$ is an $M$-matrix, then $\sf M(\c,\d)^\#$ is an $M$-matrix.
\item If $\displaystyle\sum\limits_{k=1}^n\frac{c_k^2}{d_k^2}=\dfrac{d^{**}}{d^*}$, then $\sf M(\c,\d)^\#$ is an $M$-matrix if and only if $(\alpha\DD_{\d}-\c\c^\top)^\#$ is an $M$-matrix.
\end{enumerate}
\end{propo} 

\begin{corollary}
Given $\c,\d>0$, the following statements hold:
\begin{enumerate}[{\rm (i)}]
\item If $||\c||\ge \sqrt{d^*d^{**}}$ and $\sf M(\c,\d)^\#$ is an $M$-matrix, then $(\alpha\DD_{\d}-\c\c^\top)^\#$ is an $M$-matrix.
\item If $||\c||\le d_*\sqrt{\dfrac{d^{**}}{d^*}}$ and $(\alpha\DD_{\d}-\c\c^\top)^\#$ is an $M$-matrix, then $\sf M(\c,\d)^\#$ is an $M$-matrix.
\end{enumerate}
 \end{corollary}

When $\sf d=c$, that is, in the Laplacian setting, we can recover the $M$-property of a recoverable complete network $\Gamma$ from the $M$-property of a star network. The converse implication, however, does not generally hold (see Remark \ref{ctos}). 

\begin{corollary}\label{recolapla}
    Let $\c>0$ and $0<t\le \sqrt{n}$. If $\MM(\c,t\c)^\#$ is an $M$-matrix, then $(\alpha \DD_{\c}-\c\c^\top)^\#$ is an $M$-matrix, where $\alpha=\sum\limits_{k=1}^n c_k.$
\end{corollary}
\proof Taking $\d=t\c$, the hypothesis of Proposition \ref{equivalence} (i) or (iii) is satisfied. Hence, the result.\qed

\begin{rem}\label{ctos}
    The converse of Corollary \ref{recolapla} is not true. For instance taking $\c=(1,2,1)^\top$ and the matrix $$\alpha \DD_\c-\c\c^\top=\begin{bmatrix}
        3 & -2 & -1\\
        -2 & 4 & -2\\
        -1 & -2 & 3
    \end{bmatrix},$$  its group inverse is
    $\frac{1}{72}\begin{bmatrix}
        11 & -4 & -7\\
        -4 & 8 & -4\\
        -7 & -4 & 11
    \end{bmatrix}$ which is an $M$-matrix, since it is a $Z$-matrix having all nonnegative principal minors.
    But, $\MM(\c,\c)=\begin{bmatrix}
        4 & -1 & -2 & -1\\
        -1 & 1 & 0 & 0\\
        -2 & 0 & 2 & 0\\
        -1 & 0 & 0 & 1
    \end{bmatrix},$ has the group inverse $\frac{1}{32}\begin{bmatrix}
        5 & -3 & 1 & -3\\
        -3 & 21 & -7 & -11\\
        1 & -7 & 13 & -7\\
        -3 & -11 & -7 & 21
    \end{bmatrix}$, which is not an $M$-matrix.
    
\end{rem}

\subsection*{Acknowledgements}
 Sweta Patra thanks the Office of Global Engagement, IIT Madras for the partial financial assistance to visit Universitat Polit\`ecnica de Catalunya, Barcelona through the IIE program, as well as the Departament de Matem\`atiques, UPC. Part of her work was performed during her visit, and she thanks the first two authors for their excellent hospitality. This work has been partially supported by the Spanish
	Research Council (Ministerio de Ciencia e Innovaci\'on) under project 
	PID2021-122501NB-I00, and  by the Universitat Polit\`ecnica de Catalunya through the AGRUPS-UPC 2025 funds.

\subsection*{Conflict of interest}
 The authors declare that they have no known competing financial interests or personal relationships that could have appeared to influence the work reported in this paper.

\end{document}